\documentclass{amsart}
\usepackage{setspace, amsmath, amsthm, amssymb, amsfonts, amscd, epic, graphicx, ulem, dsfont}
\usepackage[T1]{fontenc}
\usepackage{multirow}
\usepackage{bbm}
\usepackage{enumerate}

\makeatletter \@namedef{subjclassname@2010}{
  \textup{2010} Mathematics Subject Classification}
\makeatother

\newtheorem{thm}{Theorem}[section]
\newtheorem{cor}[thm]{Corollary}
\newtheorem{lem}[thm]{Lemma}
\newtheorem{pro}[thm]{Proposition}

\theoremstyle{remark}
\newtheorem*{rema}{Remark}

\theoremstyle{definition}

\newcommand{\Ima}{\text{\rm{Im}}}
\newcommand{\Real}{\text{\rm{Re}}}

\newcommand{\R}{\mathbb{R}}
\newcommand{\Z}{\mathbb{Z}}
\newcommand{\N}{\mathbb{N}}

\begin{document}

\title[Roots of operators]{On the Existence of Normal Square and Nth Roots of Operators}
\author[M. H. MORTAD]{Mohammed Hichem Mortad}

\dedicatory{}
\thanks{}
\date{}


 \address{ Department of
Mathematics, University of Oran 1, Ahmed Ben Bella, B.P. 1524, El
Menouar, Oran 31000, Algeria.\newline {\bf Mailing address}:
\newline Pr Mohammed Hichem Mortad \newline BP 7085 Seddikia Oran
\newline 31013 \newline Algeria}

\email{mhmortad@gmail.com, mortad@univ-oran.dz.}

\begin{abstract}
The primary purpose of this paper is to show the existence of normal
square and nth roots of some classes of bounded operators on Hilbert
spaces. Two interesting simple results hold. Namely, under simple
conditions we show that if any operator $T$ is such that $T^2=0$,
then this implies that $T$ is normal and so $T=0$. Also, we will see
when the square root of an arbitrary bounded operator is normal.
\end{abstract}

\maketitle

\section{Introduction}
Let $H$ be a complex Hilbert space, and let $B(H)$ denote the
algebra of all bounded linear operators on $H$.

By a square root of $A\in B(H)$, we mean a $B\in B(H)$ such that
$B^2=A$. An $A\in B(H)$ is called positive if
\[<Ax,x>\geq0,~\forall x\in H.\]

It is known that any positive operator $A$ has a unique positive
square root (which is denoted  by $\sqrt A$ or $A^{\frac{1}{2}}$ and
these notations are exclusively reserved to the unique positive
square root). However, there are non normal $N$ such that $N^2$ is
normal. As an extreme example, just consider any non-normal $N$ such
that $N^2=0$. In fact, the identity $2\times 2$ matrix $I$ has an
infinite number of \textit{self-adjoint} square roots. Indeed, the
self-adjoint
\[A_x=\left(
        \begin{array}{cc}
          x & \sqrt{1-x^2} \\
          \sqrt{1-x^2} & -x \\
        \end{array}
      \right)
\]
represents a square root of $I$ for each $x\in [-1,1]$.

There are some quite known research papers on this topic. For
instance, readers may wish to consult:
\cite{Conway-Morrel-Roots-Logarithms}, \cite{Kurepa-normal-root},
\cite{Putnam-sq-rt-normal}, \cite{Putnam-square-rt-s-a-logarithm},
\cite{Radjavi-Rosenthal-sq-roots-normal} and \cite{Stampfli-roots}.
In particular, Putnam \cite{Putnam-sq-rt-normal} gave a condition
guaranteeing that the square root of a normal operator be normal.
The problem considered in this paper is of this sort.

Recall that any $T\in B(H)$ is expressible as $T=A+iB$ where $A,B\in
B(H)$ are self-adjoint. As it is customary, we denote $A$ by $\Real
T$ and $B$ by $\Ima T$.

It is also known (cf. \cite{Mortad-Oper-TH-BOOK-WSPC}) that $T=0$
iff $A=B=0$. Therefore, if $A+iB=C+iD$, then $A=C$ and $B=D$
whenever $A,B,C,D$ are self-adjoint. It is also easily verifiable
that $T$ is normal iff $AB=BA$.

We also recall some known results which will be called on below
(these are standard facts, see \cite{Mortad-Oper-TH-BOOK-WSPC} for
proofs).

\begin{thm}\label{theorem squrae root first} Let $A,B\in B(H)$ be self-adjoint. Then:
\[0\leq A\leq B\Longrightarrow \sqrt A\leq \sqrt B\]
\end{thm}

\begin{pro}\label{absssssssssssssssss}
Let $A,B\in B(H)$ be self-adjoint. Then
\[|A|\leq B\Longrightarrow -B\leq A\leq B,\]
where $|A|=\sqrt{A^*A}$.
\end{pro}

In the end, we assume that readers are familiar with other notions
and results on $B(H)$.

\section{Existence of Normal Roots}

\begin{thm}\label{THM improvment PUTNAM}
Let $C\in B(H)$ be self-adjoint and let $T=A+iB\in B(H)$ be such
that $T^2=C$.
\begin{enumerate}
  \item If $\sigma(A)\cap \sigma(-A)=\varnothing$, then $T$ is
  self-adjoint and invertible.
  \item If $\sigma(B)\cap \sigma(-B)=\varnothing$, then $T$ is
  skew symmetric (that is, $T^*=-T$) and invertible.
\end{enumerate}
\end{thm}

The proof is based upon the following known result:

\begin{lem}\label{Rosenblum AX-XB=C THM} (\cite{rosenblum: bx-ax=c} or \cite{Bhatia-Rosenthal-OPER EQUATION})
Let $A,B\in B(H)$ be such that $\sigma(A)\cap\sigma(B)=\varnothing$.
Then the equation $AX-XB=S$ has a unique solution $X$ (in $B(H)$)
for each $S\in B(H)$.
\end{lem}

Now, we prove Theorem \ref{THM improvment PUTNAM}:

\begin{proof}
We have
\[T^2=C\Longleftrightarrow T^2=A^2-B^2+i(AB+BA)=C.\]
By the self-adjointness of $A$ and $B$, we obtain the
self-adjointness of $A^2-B^2$ and $AB+BA$ as well. Hence
\[\left\{\begin{array}{c}
                                           A^2-B^2=C, \\
                                           AB+BA=0.
                                         \end{array}
\right.\]
\begin{enumerate}
  \item If $\sigma(A)\cap \sigma(-A)=\varnothing$, then Lemma
  \ref{Rosenblum AX-XB=C THM} says that the equation
  \[AB-B(-A)=AB+BA=0\] has a unique
  solution which is necessarily $B=0$. Hence $T=A$ is self-adjoint.
  If $0\in\sigma(T)$, then $0\in\sigma(A)$ and so $0\in\sigma(-A)$
  too. This, however, would violate the assumption $\sigma(A)\cap
  \sigma(-A)=\varnothing$. Therefore, $T$ is invertible.
  \item When $\sigma(B)\cap \sigma(-B)=\varnothing$, a similar
  method yields $A=0$, i.e. $T=iB$. Hence $T$ is invertible
  because $\sigma(B)\cap \sigma(-B)=\varnothing$.
\end{enumerate}
\end{proof}

A result by Embry \cite{Embry Simil.} may be readjusted as follows:

\begin{lem}
Let $A,B\in B(H)$ be normal and such that $AB=-BA$. Designate the
numerical range of $A$ by $W(A)$. If $\sigma(A)\cap
\sigma(-A)=\varnothing$ or $0\not\in W(A)$ (resp. $\sigma(B)\cap
\sigma(-B)=\varnothing$ or $0\not\in W(B)$), then $B=0$ (resp.
$A=0$).
\end{lem}

Using the same method as above and the foregoing lemma, we may
establish:

\begin{pro}
Let $C\in B(H)$ be self-adjoint and let $T=A+iB\in B(H)$ be such
that $T^2=C$. If $0\not\in W(A)$ (or $0\not\in W(B)$), then $T$ is
self-adjoint. As above, in the first case $T=A$ and in the second
$T=iB$.
\end{pro}

We know that if $T\in B(H)$, then $T^2=0$ does not, in general,
imply that $T=0$. We also know that if $T$ satisfies
$\|T^2\|=\|T\|^2$ (for example, if $T$ is self-adjoint or normal or
normaloid in general i.e. $\|T^n\|=\|T\|^n$ for all $n$), then
$T^2=0$ does imply that $T=0$. The following result is therefore of
interest.

\begin{pro}\label{666}
Let $T\in B(H)$ be such that $T^2=0$. If $\Real T\geq 0$ (or $\Ima
T\geq0$), then $T$ is normal and so $T=0$.
\end{pro}

\begin{proof}
Write $T=A+iB$ where $A,B\in B(H)$ are self-adjoint where $A=\Real
T$ and $B=\Ima T$. Then clearly
\[T^2=A^2-B^2+i(AB+BA).\]
So, if $T^2=0$, then
\[A^2-B^2+i(AB+BA)=0\Longrightarrow \left\{\begin{array}{c}
                                           A^2=B^2, \\
                                           AB=-BA.
                                         \end{array}
\right.\] Hence, if $A\geq0$ (a similar argument works when
$B\geq0$), then
\[AB=-BA\Longrightarrow A^2B=-ABA=BA^2\Longrightarrow AB=BA.\]
 Therefore, $T$ is
normal. Accordingly,
\[\|T\|^2=\|T^2\|=0\Longrightarrow T=0,\]
as suggested.
\end{proof}

\begin{cor}
Let $T\in B(H)$ be such that $T^2=0$. If any of $\sigma(\Real T)$ or
$\sigma(\Ima T)$ is a subset of either $\R^+$ or $\R^-$, then $T$ is
normal and so $T=0$.
\end{cor}

\begin{rema}
The condition $\Real T\geq 0$ (or $\Ima T\geq0$) in Proposition
\ref{666} is, in general, not sufficient to make $T$ normaloid.
Indeed, if $V$ is the Volterra operator on $L^2[0,1]$ say, then (cf.
\cite{Mortad-Oper-TH-BOOK-WSPC}) $V$ is not normaloid as
\[r(V)=0\neq \frac{2}{\pi}=\|V\|\]
where $r(V)$ denotes the spectral radius of $V$. It can, however,
easily be checked that $\Real V\geq0$.
\end{rema}

We finish this section with a general criterion guaranteeing the
normality of the square root (this generalizes Proposition
\ref{666}):

\begin{thm}
Let $S=C+iD\in B(H)$ be such that $T^2=S$ where $T=A+iB\in B(H)$.
Then (if $[\cdot,\cdot]$ denotes the usual commutator)
\[[B,C]=[A,D]\]
\text{ and }
\[[A,C]=[B,D].\]
(Consequently, $BC=CB\Leftrightarrow AD=DA$ and
$AC=CA\Leftrightarrow BD=DB$).

 If $A\geq0$ (or $A\leq0$), then
\[T\text{ is normal }\Longleftrightarrow AD=DA.\]

If $B\geq0$ (or $B\leq0$), then
\[T\text{ is normal }\Longleftrightarrow BD=DB.\]
In particular, if $S$ is self-adjoint (i.e. $D=0$) and $A\geq0$ or
$A\leq0$ or $B\geq0$ or $B\leq 0$, then $T$ is always normal.
\end{thm}

\begin{proof}
By assumption,
\[A^2-B^2+i(AB+BA)=C+iD\Longrightarrow \left\{\begin{array}{c}
                                           A^2-B^2=C, \\
                                           AB+BA=D.
                                         \end{array}
\right.\] Hence
\[AB+BA=D\Longrightarrow A^2B+ABA=AD\text{ and }ABA+BA^2=DA\]
and so
\[A^2B-BA^2=AD-DA.\]
Since $A^2=B^2+C$, we get
\[BC-CB=AD-DA.\]

Also,
\[AB+BA=D\Longrightarrow AB^2+BAB=DB\text{ and }BAB+B^2A=BD\]
and so as above
\[B^2A-AB^2=BD-DB\]
and by invoking $B^2=A^2-C$, we obtain
\[AC-CA=BD-DB.\]

To show the last two assertions, we have from above that
\[A^2B=BA^2\Longleftrightarrow AD=DA \text{ and } B^2A=AB^2\Longleftrightarrow BD=DB.\]
So if $A\geq0$, then clearly
\[A^2B=BA^2\Longleftrightarrow AB=BA,\]
and the previous holds iff $T=A+iB$ is normal. A similar reasoning
applies when $A\leq0$. Finally, argue similarly in the event
$B\geq0$ or $B\leq0$ and this completes the proof.
\end{proof}

\section{Explicit Construction of Roots}

It is known (cf. \cite{RUD}) that if $N$ is normal with a spectral
integral $\int_{\sigma(N)}\lambda dE$ where $E$ is a spectral
measure, then $\int_{\sigma(N)}\sqrt\lambda dE$ is square root of
$N$ where $\sqrt\lambda$ is a complex square root of $\lambda$.

Under an extra condition on the normal operator, we can have an even
more explicit formula.

\begin{thm}\label{normal square root principal THM}
Let $N=C+iD\in B(H)$ be normal with either $D\geq 0$ or $D\leq 0$
(equivalently, $\sigma(D)\subset\R^+$ or $\sigma(D)\subset \R^-$).
When $D\geq 0$, then
\[T=\left(\frac{|N|+C}{2}\right)^{\frac{1}{2}}+i\left(\frac{|N|-C}{2}\right)^{\frac{1}{2}}\]
is a normal square root of $N$. If $D\leq 0$, then
\[T=\left(\frac{|N|+C}{2}\right)^{\frac{1}{2}}-i\left(\frac{|N|-C}{2}\right)^{\frac{1}{2}}\]
is another normal square root of $N$.
\end{thm}

\begin{proof}
The hardest part of the proof is the meticulousness! First $C$ and
$D$ are self-adjoint. Besides, $CD=DC$ as $N$ is normal. Then
\[|N|^2=N^*N=(C-iD)(C+iD)=C^2+D^2\geq C^2\]
as $D^2\geq0$ because $D$ is self-adjoint. By Theorem \ref{theorem
squrae root first}, we get $|N|\geq |C|$. Hence, by Proposition
\ref{absssssssssssssssss}
\[-|N|\leq C\leq |N|\]
and so
\[|N|-C\geq 0\text{ and } |N|+C\geq0.\]
Therefore, it makes sense to define their \textit{positive} square
roots. Consider (the self-adjoint!)
\[A=\left(\frac{|N|+C}{2}\right)^{\frac{1}{2}}\text{ and } B=\left(\frac{|N|-C}{2}\right)^{\frac{1}{2}}.\]
Since $|N|$ commutes with $C$, it follows that $A$ commutes with
$B$. Consequently, the operator $M:=A+iB$ is normal. Finally,
\begin{enumerate}
  \item If $D\geq 0$, then
  \begin{align*}
  M^2&=(A+iB)(A+iB)\\&=A^2-B^2+i(AB+BA)\\
  &= \frac{|N|+C}{2}-\frac{|N|-C}{2}+2i\left(\frac{|N|^2-C^2}{4}\right)^{\frac{1}{2}}\\
  &=C+i(D^2)^{\frac{1}{2}}\\
  &=C+iD\\
  &=N,
  \end{align*}
  that is $M$ is a normal square root of $N$.
  \item A similar argument applies when $D\leq 0$. In this case,
  \[(M^*)^2=N,\]
  that is, $M^*$ is a normal square root of $N$. This marks the end
  of the proof.
\end{enumerate}
\end{proof}

This approach, besides its explicit construction, does apply for
higher powers of the type $2^n$. In other language, the algorithm
prescribed in the previous proof may be applied to deal with
biquadratic equations or in general equations of the form
$T^{2^{n}}=N$ where $n\in\N$. We have:

\begin{cor}
Let $N=C+iD\in B(H)$ be normal with either $D\geq 0$ (or $D\leq 0$).
Let $T\in B(H)$ be such that $T^4=N$. Then a normal 4th root of $T$
is given by
\[T=\left(\frac{|S|+\Real S}{2}\right)^{\frac{1}{2}}+i\left(\frac{|S|-\Real S}{2}\right)^{\frac{1}{2}}\]
where
\[S=\left(\frac{|N|+C}{2}\right)^{\frac{1}{2}}+i\left(\frac{|N|-C}{2}\right)^{\frac{1}{2}}\]
\end{cor}

\begin{proof}
Put $S=T^2$. Then $T^4=N$ becomes $S^2=N$. So, if e.g. $N$ is such
that $\sigma(\Ima N)\subset\R^+$, then $S$ is normal. Moreover, by
Theorem \ref{normal square root principal THM}, we know that
\[S=\left(\frac{|N|+C}{2}\right)^{\frac{1}{2}}+i\left(\frac{|N|-C}{2}\right)^{\frac{1}{2}}\]
and it is normal. Since clearly
\[\Ima S=\left(\frac{|N|-C}{2}\right)^{\frac{1}{2}}\geq0,\]
it follows that $T^2=S$ has a normal solution $T$ given by
\[T=\left(\frac{|S|+\Real S}{2}\right)^{\frac{1}{2}}+i\left(\frac{|S|-\Real S}{2}\right)^{\frac{1}{2}}.\]
\end{proof}

\begin{rema}
By induction, we know how to find a root of order $2^n$ of normal
operators very explicitly.
\end{rema}

We finish with the case of general nth roots. First, the following
result should be readily verified.

\begin{lem}\label{lemma e iA+2kpi I}
Let $A\in B(H)$. If $k\in\Z$, then
\[e^{i(A+2k\pi I)}=e^{iA}.\]
\end{lem}

\begin{thm}
Let $N\in B(H)$ be normal. Let $n\in\N$. Then $N$ has always an nth
root which is also normal and given by
\[N^{\frac{1}{n}}=|N|^{\frac{1}{n}}e^{i\left(\frac{A+2k\pi I}{n}\right)}\]
for some self-adjoint $A\in B(H)$ and where $k\in\Z$.
\end{thm}

\begin{proof}Since $N$ is normal, it may be expressed as (see e.g.
\cite{RUD})
\[N=UP=PU,\]
where $U$ is unitary and $P=\sqrt{N^*N}=|N|$. Since $U$ is unitary,
$U=e^{iA}$ for some self-adjoint $A\in B(H)$ (cf. Proposition 18.20
in \cite{MV}). Hence
\[N=e^{iA}P=Pe^{iA}.\]
Set for $k\in\Z$
\[M=P^{\frac{1}{n}}e^{i\left(\frac{A+2k\pi I}{n}\right)}\]
 where $P^{\frac{1}{n}}$ is the unique positive
nth root of $P$. Then $M$ is normal (cf.
\cite{Mortad-Oper-TH-BOOK-WSPC}) because it is a product of two
commuting normal operators (in fact, a product of a positive
operator and a unitary one). By the commutativity of the factors and
simple results, we obtain
\[M^n=\left(P^{\frac{1}{n}}e^{i\left(\frac{A+2k\pi I}{n}\right)}\right)^n=Pe^{i(A+2k\pi I)}=Pe^{iA}=N,\]
as needed.
\end{proof}

\begin{rema}
As is the case of many results on normal operators, the previous two
theorems too are inspired by results about complex numbers. This
corroborates and strengthens the ressemblance which is already known
to readers.
\end{rema}

\end{document}